\documentclass[a4paper]{article}
\usepackage{latexsym}
\usepackage{amsmath,amssymb,amsfonts}
\renewcommand{\baselinestretch}{1.2}
\DeclareMathOperator{\codim}{codim}
\newtheorem{prethm}{{\bf Theorem}}

\newenvironment{thm}{\begin{prethm}{\hspace{-0.5
               em}{\bf.}}}{\end{prethm}}

\newtheorem{prepro}[prethm]{Proposition}

\newtheorem{prelem}[prethm]{Lemma}
\newenvironment{lem}{\begin{prelem}{\hspace{-0.5
               em}{\bf.}}}{\end{prelem}}

\newtheorem{precor}[prethm]{Corollary}
\newenvironment{cor}{\begin{precor}{\hspace{-0.5
               em}{\bf.}}}{\end{precor}}

\newtheorem{preremark}{{\bf Remark}}

\newenvironment{rem}{\begin{preremark}\em{\hspace{-0.5
              em}{\bf.}}}{\end{preremark}}

\newtheorem{preexample}{{\bf Example}}

\newenvironment{example}{\begin{preexample}\em{\hspace{-0.5
               em}{\bf.}}}{\end{preexample}}

\newtheorem{preproof}{{\bf Proof.}}

\newenvironment{proof}[1]{\begin{preproof}{\rm
               #1}\hfill{$\Box$}}{\end{preproof}}

\renewcommand{\thefootnote}

\title{Focal Varieties of Curves of Genus 6 and 8}
\author{A. Bajravani\\
{\footnotesize {Department of Pure Mathematics, Faculty of mathematical sciences,}}\\
{\footnotesize {Tarbiat Modares University, Tehran, Iran }}\\
{\footnotesize {P. O. Box 14115-134}}}
\begin{document}
\footnotetext{E-mail Address: {\tt bajravani@modares.ac.ir}}
\date{}
\maketitle
\begin{quote}
{\small \hfill{\rule{13.3cm}{.1mm}\hskip2cm}
\textbf{Abstract}\vspace{1mm}

{\renewcommand{\baselinestretch}{1}
\parskip = 0 mm
In this paper we give a simple Torelli type theorem for curves of
genus $6$ and $8$ by showing that these curves can be reconstructed
from their Brill Noether varieties. Among other results, it is shown
that the focal variety of a general, canonical and nonhyperelliptic
curve of genus $6$,
is a hypersurface.\\

}

}

\vspace{-3mm}\hfill{\rule{13.3cm}{.1mm}\hskip2cm}
\end{quote}

\section{Introduction}

In a series of papers, \cite{C-S 1}, \cite{C-S 2} and \cite{C-S 3},
C. Ciliberto and E. Sernesi showed that a canonical,
non-hyperelliptic curve of genus $g$, can be reconstructed from its
Brill-Noether varieties. They proved that $C^{1}_{g-1}$
parameterizes a family of linear subspaces of ${\mathbf{P}^{g-1}}$,
which can determine the canonical curve, uniquely ( See \cite{C-S 2}
). A canonical curve is a non-hyperelliptic curve of genus $g$ which
is embedded in $\mathbf{P}^{g-1}$ by its canonical line bundle. In
these papers they use the technique of first order foci. Briefly,
first order foci is the locus where the rank of some special
morphism drops. For a general element of the parameter space in
\cite{C-S 2}, first order foci is a rational normal curve in some
linear subspace of $\mathbf{P}^{g-1}$. The union of these rational
curves, called {\em focal variety}, contains the canonical curve and
specifically in genus $5$ case, it is actually a covariant of a net
of quadrics in $\mathbf{P}^{g-1}$ (See \cite{C-S 3}). In
prolongation of the works \cite{C-S 1}, \cite{C-S 2} in \cite{C-S
3}, for a general canonical curve of genus $g=2n+1$, they use
$C^{1}_{n+2}$, where they could use second order foci, and they got
a Torelli type theorem for general curves of odd genus (See
 \cite{C-S 3}). In that work, they show that the focal variety of a
 general nonhyperelliptic and canonical curve of genus $5$  is a
hypersurface. We show that this fact is valid also
for genus $6$ case. The fact that, the focal variety of a general
curve of genus $6$ is of expected dimension, implies that our
computation of second order foci in genus $6$ case is valid.

Throughout this paper, we will work over an algebraically closed
field of characteristic zero. Assume that $C$ is a general canonical
non-hyperelliptic curve of genus $6$ (res. of genus 8), which is
embedded by its canonical line bundle in $\mathbf{P}^{5}$ (res. in
$\mathbf{P}^{7}$ ). As it is usual in Brill-Noether theory,
\noindent $\mathbf{g}^{r}_{d}$ denotes a line bundle of degree $d$
with global sections of dimension at least $r+1$. We denote by
$W^{r}_{d}$ the set of $\mathbf{g}^{r}_{d}$'s, and we set
$W_d=W^0_d$. It is known in Brill-Noether theory that $W^{r}_{d}$
has a scheme structure and for a general curve it is an irreducible
nonsingular variety, unless  it has dimension zero. We will denote
the set of effective divisors of degree $d$ with global sections of
dimension at least $r+1$ by $C^{r}_{d}$. The morphism: $\alpha_{d}:
C_{d} \rightarrow W_{d}$ is the Abel- Jacobi map. Also,
$\mathbf{M}_{g}$ denotes the moduli space of curves of genus $g$.

This paper consists of three sections. In the section $2$, we prove
that the focal variety of a general canonical and non-hyperelliptic
curve of genus $6$,  is a hypersurface. Then, by describing the second
order foci in this case,
 we will give a short proof of a Torelli type
theorem for such a curve. We see that a general canonical
nonhyperelliptic curve of genus $6$, is an irreducible component of
the closure of the union of its second order foci. We also give examples to
show that some components of the focal scheme may be of lower
dimension than expected, for special curves. In section $3$, for
general canonical and nonhyperelliptic curves of genus $8$, we prove
that the  first order foci are, in general, an irreducible cubic
threefold in a linear subspace of dimension 4: in one of the
possible cases, their singular locus is a rational quartic curve and
the first order foci is the secant variety of the rational quartic
curve. Also, using similar argument as in $ \cite{C-S 2}$ and
\cite{C-S 3},
 we show that the union of a certain $1$-parameter family of these rational curves, is
 a rational surface, which is used to prove Torelli with this approach.

\section{The genus 6 case}

\subsection{First Order Foci and Focal Variety}

Consider the scheme $W^{1}_{5}$, which   is a $2$-dimensional
variety for a general curve of genus $6$. By generality condition
the Abel-Jacobi map is a $\mathbf{P}^{1}$-bundle. So $C^{1}_{5}$ is a 3-fold,
which we denote by $S$. For a divisor $D_{s} \in S $, let
$\Lambda_{s}=\overline{D_{s}}$ be the three dimensional linear space
spanned by $D_{s}$. These linear subspaces   fit together to make a
$3$-dimensional family of
 $5$-secant linear subspaces of $\mathbf{P}^{5}$ parameterized by $S$:

\begin{equation}\label{firstfoci}
\begin{array}{ccc}
  \Lambda & \subset& S\times \mathbf{P}^{5}\\
   \downarrow & \pi & \\
   S&  &
\end{array}
\end{equation}

 For more details about first order
 foci, see \cite{C-S 2} and \cite{C-S 3}.

\begin{thm}\label{1}
For a general $L \in W^{1}_{5}$ and for all sufficiently general
$s\in \alpha ^{-1}_{5}(L)$,  the focal variety $F_{s}$ at $s \in S
$, is a rational normal cubic curve in $\Lambda_s$.
\end{thm}
\begin{proof}{
 By the Brill-Noether theorem for a general line bundle $L\in
W^{1}_{5}$ on a general curve, the linear series $|L|$ is base point
free (See \cite{G-H}). Now, the theorem follows by $2.5$ of
\cite{C-S 2}. }
\end{proof}

\begin{rem}
Note that, Theorem $1$ is valid without generality assumption on
the line bundle $L$, and we will come back to its proof in Theorem
$7$. But for our aim in this section it is enough to take a general
line bundle $L\in W^{1}_{5}$.
\end{rem}

 For a line
bundle $L \in W^{1}_{5}$, the irreducible variety $ Q=\cup_{s\in
\alpha_{5}^{-1}(L)}\Lambda_{s} $ is a quadric hypersurface
 with vertex a line which we denote by $\Gamma$.

\begin{lem}
For any $s\in \alpha_{5}^{-1}(L)$, the rational
normal curve $ F_{s} $ cuts the vertex $ \Gamma $ in two points.
\end{lem}

\begin{proof}{
This has been proved originally in
proposition $4.2$ of \cite{C-S 2}.
}
\end{proof}

 For a fixed
line bundle $L \in W^{1}_{5}$, the union $\cup_{s\in
\alpha_{5}^{-1}(L)}F_{s}$ is
 birational with $\mathbf{P}^{1} \times \mathbf{P}^{1}$ and under this
birationality, the lines $\{pt\}\times \mathbf{P}^{1}$ are mapped to focal
curves (see  \cite{C-S 2}). Let $F_{C}=\overline{\cup_{s \in S
}F_{s}}$ be the focal variety of the canonical curve $C$. We prove
that $F_{C}$ is a hypersurface.
\begin{lem} Let $\mathbf{M}_{g}$ be the moduli space of curves of genus g.
Define a map:
\begin{eqnarray*}
 \Psi:&\mathbf{M}_{g}&\rightarrow \mathbb{Z}  \\
  \Psi([C])\!\!\!\!\!\!&=&\!\!\!\!\!\!\dim(F_{C})
\end{eqnarray*}

Then $\Psi$ is a lower-semicontinuous function.
\end{lem}
\begin{proof}{
This follows from the fact that $\dim(F_C) = \dim(V(\chi_C)) -
\dim(V(X_{C})_{t})$ where $V(\chi_C) \subset S \times
\mathbf{P}^{g-1}$ is the focal locus, $f_C: V(\chi_C) \to
\mathbf{P}^{g-1}$ is the projection and $V(X_{C})_{t}$ is a general
fibre of $f_{C}$ for a general $t \in \mathbf{P}^{g-1}$. Now, we
have a map $\phi: \cup_{C\in \mathbf{M}_g} V(\chi_C)  \to
\mathbf{P}^{g-1} \times \mathbf{M}_{g}$ which restricts to $f_C$
over $C$. By \cite[page 95]{Har}, we know that the dimension of the
general fibre of $\phi$ is upper-semicontinuous as a function of
$C$, which gives the upper-semicontinuity of the general fibre of
$f_{C}$. Now it follows that $\dim(F_{C})$ is lower semicontinuous
as a function of $C$. }
\end{proof}

 Note that using this lemma it is enough to find a special curve whose
focal variety has a component which is a hypersurface. The following theorem gives an
answer to this question.
\begin{thm}
Let $C$ be a smooth plane quintic. Then $F_{C}$ is a
hypersurface.
\end{thm}
\begin{proof}{
 For a smooth plane quintic curve $C$, a general element of $S$ is of the form:
 $$D_{s}=p_{1}+p_{2}+p_{3}+p_{4}+q$$
where $p_{1}+\cdots+p_{5}$ is a general element
 of the $g_{5}^{2}$ and $q$ is a general element of $C$.

 The Veronese embedding maps the line passing through
 $p_{1}+\cdots+p_{5}$  to a conic $\Gamma$ in $\mathbf{P}^{5}$. Also, it maps  the tangent line to $C$ at $q$, to
   a conic $\overline{\Gamma}$ and we have
  $\Gamma \cap \overline{\Gamma} =\left\{p\right\}$, and  the focal curve is $\Gamma \cup \overline{p,q}$.
  So visibly $F_{C}$ consists of two parts. One is a
  union of conics $\Gamma$, which is two dimensional,
  because all of them are contained in the Veronese surface.
  The other part is a  union of lines, which we want
  to prove it is a fourfold.
The above argument proves that the focal variety consists of all the
"join" of the veronese surface with C (i.e. the closure of the union
of lines spanned by a general point q of C and a general point of
the Veronese). Now, the lines in this family through q form a cone
over the Veronese with vertex q. So their union has dimension 3.
Now, moving $q$ on $C$ leads us to a $1$-dimensional
 family of $3$-dimensional cones over the Veronese surface, with vertices moving on $C$.
  Therefore we will have $ \dim(F_{C})=4 $. }
 \end{proof}

Consider that in this case the canonical curve $C$ and the Veronese
surface $F$ are contained in the focal variety $F_{C}$. Now by Lemma
$3$ and Theorem $4$, we obtain the following

\begin{thm}
Let $C$ be a general canonical and nonhyperelliptic curve, of genus $6$.
Then its focal variety, $F_{C}$, is a hypersurface in
$\mathbf{P}^{5}$.
\end{thm}

Note that, without generality assumption on the curve $C$,
Theorem $5$ is not valid. In general
the following examples show that some
components of the focal scheme may be of lower dimension than expected.

\begin{example}
Let $C$ be a bielliptic curve of genus $g$. In this case
$C^{1}_{g-1}$ has a component $S$, whose general member is of the
form $D_{s}= P_{1}+P_{2}+Q_{1}+Q_{2}+R_{1}+\cdots+R_{g-5}$, where
the lines $\overline{P_{1}P_{2}}$ and $\overline{Q_{1}Q_{2}}$ are
two generators of the elliptic cone, and $R_{1},\ldots,R_{g-5}$ are
general points of $C$. In this setting we have
\begin{center}
$F_{s}=\overline{P_{1}P_{2}}\cup \overline{Q_{1}Q_{2}} \cup
\overline{OR_{1}}\cup\cdots\cup \overline{OR_{g-5}}$
\end{center}
For more details about constructions of focal loci in exceptional
cases see \cite{C-S 2}. Now it easy to see that the focal variety is
of dimension $2$. In fact by taking $P_{1}, P_{2}, Q_{1}, Q_{2}$
fixed and by moving $R_{i}$, ($1\leq i \leq g-5$), on the curve we
will get a cone over $C$ with vertex $O$, which  is of dimension
$2$.\\
In bielliptic case, $C^{1}_{g-1}$ has a component whose general element is the residual of the general
element of S with respect to the canonical series.
 It is not known, whether or not, the focal variety over this component is a hypersurface.
\end{example}
\begin{example}
Let $C$ be a trigonal curve of genus $g \ge 6$. In this case
$C^{1}_{g-1}$ has a component $S$, whose general member has the form
$D_{s}=P+Q+R+P_{1}+\cdots+P_{g-4}$, where $P+Q+R \in
\mathbf{g}^{1}_{3}$ and $P_{1}+\cdots+P_{g-4}$ are general points of
$C$. Let $\underline{r}$ denotes the line joining the points $P$,
$Q$ and $R$.
 The irreducible scheme $S$ has dimension $g-3$. Take $S$ as the
parameter space of the family of linear subspaces of
$\mathbf{P}^{g-1}$ spanned by divisors $D_{s}$, in which $D_{s}$ is
a general member of $S$. As in the previous example the expected
dimension is $g-2$. In this setting we have
\begin{center}
$F_{s}= \underline{r}\cup \overline{P_{1}Q_{1}} \cup\cdots \cup
\overline{P_{g-4}Q_{g-4}} $.
\end{center}
The nature of the points $Q_{i}$ is explained completely in
\cite{C-S 2}. For our aim it is enough just to know that they are
points on the line $\underline{r}$, which is the line passing
through points $P,Q,R$. It is easy to see that the focal variety is
of dimension at most $4$. In fact for a line $L$ generated by a
general member $P+Q+R$ of $\mathbf{g}^{1}_{3}$, set
$F_{L}=\bigcup_{P_{1},\ldots,P_{g-4}}F_{P_{1},\ldots,P_{g-4}}$,
where $F_{P_{1},\ldots,P_{g-4}}= \underline{r}\cup
\overline{P_{1}Q_{1}} \cup\cdots \cup \overline{P_{g-4}Q_{g-4}}$,
and $P_{1},\ldots,P_{g-4}$ moves on the curve. Consider that $F_{L}$
is contained in the joining variety of $C$ and the line $L$, so
$F_{L}$ has dimension at most $3$. By moving $L$ in
$\mathbf{g}^{1}_{3}$ we will get a $1$-dimensional
family of $F_{L}$'s, whose union, $F_{C}$, has dimension at most $4$.\\
Also in this case $C^{1}_{g-1}$ has another  component and again it
 is not known whether or not the focal variety on this component is a hypersurface.
\end{example}

\subsection{Second order Foci in genus 6 case}

We have seen that the family (\ref{firstfoci}) defines, for $s \in
S$, a closed subscheme $F_{s}\subset \Lambda_{s}$, whose points are
the so called first order foci at $s$. All these subschemes fit
together in a closed subscheme $\mathcal{F} \subset \Lambda$ and we
obtain a morphism $\pi_{1}:\mathcal{F}\rightarrow S$ and a diagram :

\begin{equation} \label{secodfoci}
\begin{array}{ccc}
  \mathcal{F} & \subset& S\times \mathbf{P}^{5}  \\
   \downarrow & \pi_{1} &   \\
   S&  &
\end{array}
\end{equation}

By Theorem $1$ (or by Theorem $8$), it follows that for all
sufficiently general $s\in S $ the fibre $ F_{s}=\pi_{1}^{-1}(s) $,
is a rational normal curve in $\Lambda_{s}$. For such an $s$, we can
introduce the second order foci of the family (\ref{firstfoci}),
defined as the first order foci of the family (\ref{secodfoci}) at
$s$. For more details about second order foci see \cite{C-S 3}. We
will denote by $D_{1}(\xi_{s})$, the second order foci of the family
 (\ref{firstfoci}) at $s$.

\begin{lem}
The second order foci in the genus $6$ case is a divisor of the form
$2E+F$ over the focal curve, where $E=p_{1}+\cdots+p_{5}$ and $F$ is
a divisor of degree $2$. In fact $F$ is the intersection locus of
the vertex $\Gamma$ with the rational normal curve $F_{s}$.
\end{lem}
\begin{proof}{
By the same computation as in \cite{C-S 3} we have that
$\deg(D_{1}(\xi_{s})) \leq 12 $. The points $p_{1},\ldots, p_{5}$
 belong to $D_{1}(\xi_{s})$, because they belong to the two dimensional
family of linear spaces spanned by the divisors which contain the
point $p_{i}$ in their support. By the same reason the points $p$
and $q$, which   are the intersection points of $F_{s}$ with
$\Gamma$,
  belong to $D_{1}(\xi_{s})$. Now, as an application of Porteous
formula, it is easy to see that the points $p_{1},\ldots, p_{5}$
appear with multiplicity at least two. So they have multiplicity two
in $F_{s}$. Therefore $p$ and $q$ appear with multiplicity one and
we are done.}
\end{proof}
\begin{cor}
The canonical curve $k(C)$ can be constructed uniquely from the
family (\ref{firstfoci}). In fact it is a component of the closure
of the union of its second order foci.
\end{cor}

During preparing this paper, C. Ciliberto and E. Sernesi pointed out
to me the paper \cite{C-S 4}, where similar and more general
computations of second order foci have been done.

\section{The genus 8 case}

\subsection{First order foci and its singular locus}

In the genus $8$ case take $S=C_{6}^{1}$, as parameter space, which
is a threefold; so we will have $3$-dimensional family of
$\Lambda_{s}$'s inside $\mathbf{P}^{7}$. For a general $s \in S$ the
morphism:
\begin{center}
$\Phi_{s}:T_{S,s} \bigotimes {O_{\Lambda_{s}}} \rightarrow
N_{\Lambda_{s}/ \mathbf{P}^{7}}. $
\end{center}
is a morphism between vector bundles of rank 3. Consider the first
order foci $F_{s}=\{x\in \Lambda_{s}| rk(\Phi_{s}(x))\leq2\}$. For a
fixed line bundle $L \in W^{1}_{6}$, the rational normal scroll
 $H_{L}=\cup_{D_{s}\in \alpha^{-1}_{6}(L)}\Lambda_{s} $
has a line $\Gamma$ as it's vertex. The vertex $\Gamma$ is contained
in $F_{s}$, because it is contained in  the $1$-parameter family of
$\Lambda_{s}$ parameterized by $\mathbf{P}^{1}\cong \alpha_{6}^{-1}(L)$.

\begin{thm}
For a sufficiently general $s\in S$, $\Phi_{s}$ is $1$-generic. In particular $F_{s}$ is an irreducible
cubic threefold.
\end{thm}
\begin{proof}{
By generality condition on the curve,
it follows that $|L^{2}|$ is a $\mathbf{g}^{4}_{g+4}$ and
 it is not composed with an involution.
Now the proof of this theorem is similar to the proof of theorem $2$
of \cite{C-S 3} to which we refer.
 }
\end{proof}

Now set $\Gamma_{s}=\left\{x\in \Lambda_{s}| rk(\Phi_{s}(x))\leq
1\right\}$. By corollary $3.3$ of \cite{E.}, we have
$\codim_{\Lambda_{s}}\Gamma_{s}\geq 3+3-3 = 3$, or equivalently
$\dim \Gamma_{s}\leq 4-3=1$. So there are two cases. First suppose
that $\dim\Gamma_{s} = 0$. In this case as an application of
Porteous formula we find that $\deg(\Gamma_{s})=6$. Set
$D_{s}=p_{1}+\cdots+p_{6}$. Since by the same reasons as in lemma
$6$, $p_{i}$ $(1\leq i \leq 6)$, belongs to $\Gamma_{s}$, we have
that $\Gamma_{s}=\{p_{1},\ldots,p_{6}\}$. Now suppose that
 $\dim(\Gamma_{s})=1$. In this case by theorem $5.1$ of \cite{E.},
$\Gamma_{s}$ is a rational normal curve and $\Phi_{s}$ is a
Catalecticant matrix, so $F_{s}$ is the secant variety of
$\Gamma_{s}$. In the second case consider the morphism
$$\Psi_{s}:T_{S,s} \otimes {O_{\Gamma_{s}}}
\rightarrow N_{\Gamma_{s}/ \mathbf{P}^{7}} $$ which is a morphism between
vector bundles of ranks 3 and 6, respectively. Consider the locus
where the rank of $\Psi_{s}$ drops twice, i.e. the set $H_{s}=\{x\in
\Gamma_{s}|rk(\Psi_{s}(x))\leq 1\}$. By the methods of the paper
\cite{C-S 3}, we find that $\deg(H_{s})=8$. Also, we know that
$p_{i}$ $(1\leq i \leq 6)$ belong to $H_{s}$. So we will have the
following corollary, whose proof we will complete in the next
subsection.

\begin{cor}
The canonical curve $k(C)$ can be constructed uniquely from a
family, analogous to the family (\ref{firstfoci}). Precisely in the
first case, $k(C)$ is the closure of union of the singular locus of
first order focis. In the second case the canonical curve $k(C)$ is
a component of the closure of the union of those loci, where
$\Psi_{s}$'s, drops rank twice.
\end{cor}

\subsection{Focal surfaces}

For a canonical curve of genus $8$, assume that we are in the second
case and for a general line bundle $L\in W^{1}_{6}$ denote by
$\textbf{U}\subset \alpha_{6}^{-1}(L)$ the open subset such that
$\Gamma_{s}$ is a rational normal curve. Define:
$$F_{L}=\overline{\cup_{s\in \textbf{U}}\Gamma_{s}}$$
$F_{L}$ is an irreducible surface. By generality assumption the
morphism $\phi_{L^{2}}: C \rightarrow \mathbf{P}^{4}$ maps $C$ birationally
onto a curve of degree $12$. By base point free pencil trick and
generality condition on the curve $C$, the multiplication map
$$\overline{\mu}_{L}:H^{0}(C,L)\otimes H^{0}(C,L) \rightarrow
H^{0}(C,L^{2})$$
is an injective map, and so the projectivized of dual of it's image
is a point $p$ in $\mathbf{P}^{4}$ which can be realized as the vertex of a
quadric hypersurface $Q$.
 Note that $\phi_{L^{2}}(C)$ is contained in $Q$, which is a cone over a surface $V$ isomorphic to
$\mathbf{P}^{1}\times \mathbf{P}^{1}$, which has the point $p$ as it's vertex. Now by
injectivity of $\overline{\mu}_{L}$, the composition of the morphism
$\phi_{L^{2}}$ with the projection $\pi_{p}$ from the point $p$
coincides with the composition of $\phi_{L} \otimes \phi_{L}: C
\rightarrow \mathbf{P}^{1}\times \mathbf{P}^{1}$ and the Segre embedding $\nu :
\mathbf{P}^{1}\times \mathbf{P}^{1} \rightarrow \mathbf{P}^{3}$. Consider also that,
$\pi_{p}(\phi_{L^{2}})(C)$ is contained in $V$ whose generating
lines cut on $C$ the pencil $L$.
Let $\sigma : U\rightarrow V$ be the blow up of $V$ at the conductor
ideal of $\pi_{p}(\phi_{L^{2}})(C)$, and $C'\subset U$ the proper
transform of $\pi_{p}(\phi_{L^{2}})(C)$. Then the adjoint morphism
$\phi_{K_{U}+C'} : U\rightarrow \mathbf{P}^{7}$ maps $U$ onto a surface
containing $k(C)$ which contains a $1$-parameter family of rational
normal curves $L_{s}$ of degree $4$ which cut the vertex $\Gamma$ in
two points. Now, by the following lemma it follows that these curves
are the curves $\Gamma_{s}$ and $F_{L}$ is birational with
$\mathbf{P}^{1}\times \mathbf{P}^{1} $.

\begin{lem}
Let $R$ and $R'$ be rational curves in $\mathbf{P}^{4}$, which have $6$
points in common and cut a line $L$ at least in two points. Then
$R=R'$.
\end{lem}
\begin{proof}{
Let $L$ be the line which $R$ and $R'$ cut it in two points. By
projecting from $L$ to $\mathbf{P}^{2}$, we get two conics $\pi(R)$ and
$\pi(R')$ which cut each other in $6$ points. This is possible only
when $\pi(R)=\pi(R')$. But this equality means that there are
infinitely many planes $\Lambda$ which cut the rational curves $R$,
$R'$ simultaneously and contain the line $L$. Take one such a plane.
The rational curves $R$ and $R'$ have $6$ points in common and cut
the plane at least in three points. Now the lemma follows by lemma
$(4.3)$ of \cite{C-S 2}.
  }
\end{proof}

By a same method in \cite{C-S 3}, it is easy to see that the variety
$F=\cup_{s\in S}\Gamma_{s}$ should be of dimension at least $3$,
otherwise, for any line bundle $L'$, we have $F_{L}=F_{L'}$ and it
follows that the canonical curve $k(C)$ has a $\mathbf{g}_{6}^{r}$, with
$r\geq 2$, which is impossible by generality of curve. Now, the
corollary $9$ is valid. Indeed, for a general member $s$ of $S$,
$k(C)= \overline{\cup_{s\in S}\Gamma_s}$ when $\dim(\Gamma_s)=0$. Then assume that $\dim(\Gamma_s)=1$. Notice
that $\Psi_s$ does not vanishes identically since $\dim F \geq3$.
 Therefore, for a variety $H$,
$$ \overline{\cup_{s\in S}H_s}=k(C)\cup H. $$
This shows that $k(C)$ is a component of the closure of the union of those locus where $\Psi_s$ drops rank twice.

\textbf{Acknowledgment}

I am grateful to Prof. E. Sernesi, who taught and introduced me to
this topic. I benefited from warmly and energetic discussions which
I had with him, when I was  an Erasmus visitor in Rome Tre
university. I am also thankful for Prof.   C. Ciliberto, and A.
Bruno for having friendly and useful
 discussions, about this subject and many other problems, during the period.

\end{document}